\documentclass{article} 

\usepackage[ansinew]{inputenc}
\usepackage[english]{babel}
\usepackage{amsfonts}
\usepackage{amsmath}
\usepackage{amsthm}
\usepackage{latexsym}

\newtheorem{thm}{Theorem}

\newtheorem{mydef}{Definition}
\newtheorem{exa}{Example}

\title{The representer theorem for Hilbert spaces: a necessary and sufficient condition}

\author{ Francesco Dinuzzo \and  Bernhard Schölkopf }
\date{Max Planck Institute for Intelligent Systems\\
Spemannstrasse 38,\\
72076 Tübingen, Germany}

\begin{document}

\maketitle

\begin{abstract}
A family of regularization functionals is said to admit a linear representer theorem if every member of the family admits minimizers that lie in a fixed finite dimensional subspace. A recent characterization states that a general class of regularization functionals with differentiable regularizer admits a linear representer theorem if and only if the regularization term is a non-decreasing function of the norm. In this report, we improve over such result by replacing the differentiability assumption with lower semi-continuity and deriving a proof that is independent of the dimensionality of the space.
\end{abstract}

\section{Introduction}

Tikhonov regularization \cite{Tikhonov77} is a popular and well-studied methodology to address ill-posed estimation problems \cite{Wahba90}, and learning from examples \cite{Cucker01}. In this report, we focus on regularization problems defined over a real Hilbert space $\mathcal{H}$. A Hilbert space is a vector space endowed with a inner product and a norm that is complete\footnote{Meaning that Cauchy sequences are convergent.}. Such setting is general enough to take into account a broad family of finite-dimensional regularization techniques such as regularized least squares or support vector machines for classification or regression, kernel principal component analysis, as well as a variety of regularization problems defined over infinite-dimensional reproducing kernel Hilbert spaces (RKHS).

In general, we study the problem of minimizing an extended real-valued functional $J:\mathcal{H} \rightarrow \mathbb{R} \cup \{+\infty\}$ of the form
\begin{equation}\label{E01}
J(w) = f(L_1 w, \ldots, L_{\ell} w ) +\Omega(w),
\end{equation}
\noindent where $L_1, \ldots, L_{\ell}$ are bounded (continuous) linear functionals on $\mathcal{H}$. The functional $J$ is the sum of an \emph{error term} $f$, which typically depends on empirical data, and a \emph{regularizer} $\Omega$ that enforces certain desirable properties on the solution. By allowing the functional $J$ to take the value $+\infty$, problems with hard constraints on the values $L_i w$ are included in the framework.

In machine learning, the most common class of regularization problems concerns a situation where a set of data pairs $(x_i,y_i)$ is available, $\mathcal{H}$ is a space of real-valued functions, and the objective functional to be minimized is of the form
\[
J(w) = c\left((x_1,y_1,w(x_1)),\cdots, (x_{\ell},y_{\ell},w(x_{\ell})\right) + \Omega(w).
\]

\noindent It is easy to see that this setting is a particular case of (\ref{E01}). Indeed, the dependence on the data pairs $(x_i,y_i)$ can be absorbed into the definition of $f$, and $L_i$ are point-wise evaluation functionals, i.e. such that $L_i w = w(x_i)$. Several popular techniques can be cast in such regularization framework.

\begin{exa}[Regularized least squares]
Also known as ridge regression when $\mathcal{H}$ is finite-dimensional. Corresponds to the choice
\[
c\left((x_1,y_1,w(x_1)),\cdots, (x_{\ell},y_{\ell},w(x_{\ell})\right) = \gamma \sum_{i=1}^{\ell}(y_i-w(x_i))^2,
\]
\noindent and  $\Omega(w) = \|w\|^2$, where the complexity parameter $\gamma \geq 0$ controls the trade-off between fitting of training data and regularity of the solution.
\end{exa}

\begin{exa}[Support vector machine]
Given binary labels $y_i = \pm 1$, the SVM classifier can be interpreted as a regularization method corresponding to the choice
\[
c\left((x_1,y_1,w(x_1)),\cdots, (x_{\ell},y_{\ell},w(x_{\ell})\right) = \gamma \sum_{i=1}^{\ell}\max\{0, 1-y_i w(x_i)\},
\]
and $\Omega(w) = \|w\|^2$. The hard-margin SVM can be recovered by letting $\gamma \rightarrow +\infty$.
\end{exa}

\begin{exa}[Kernel principal component analysis]
Kernel PCA can be shown to be equivalent to a regularization problem where
\[
c\left((x_1,y_1,w(x_1)),\cdots, (x_{\ell},y_{\ell},w(x_{\ell})\right) = \left\{
                                                                          \begin{array}{ll}
                                                                           0, &  \frac{1}{\ell}\sum_{i=1}^{\ell}\left(w(x_i)-\frac{1}{\ell}\sum_{j=1}^{\ell}w(x_j)\right)^2=1 \\
                                                                            +\infty, & \hbox{otherwise}
                                                                          \end{array}
                                                                        \right.,
\]
\noindent and $\Omega$ is any strictly monotonically increasing function of the norm $\|w\|$ \cite{Scholkopf98}. In this problem, there are no labels $y_i$, but the feature extractor function $w$ is constrained to produce vectors with unitary empirical variance.
\end{exa}

Within the formulation (\ref{E01}), the possibility of using general continuous linear functionals $L_i$ allows to consider a much broader class of regularization problems.

\begin{exa}[Tikhonov deconvolution]
Given a input signal $u$, assume that the convolution $u \ast w$ is well-defined for any $w \in \mathcal{H}$, and the point-wise evaluated convolution functionals
\[
L_i w = (u \ast w)(x_i) = \int_{\mathcal{X}}u(s)w(x_i-s)ds,
\]
are continuous. A possible way to recover $w$ from noisy measurements $y_i$ of the “output signal” is to solve regularization problems such as
\[
\min_{w\in\mathcal{H}}\left(\gamma \sum_{i=1}^{\ell}\left(y_i - (u \ast w)(x_i)\right)^2+\|w\|^2\right),
\]
where the objective functional is of the form (\ref{E01}).
\end{exa}

\begin{exa}[Learning from probability measures]
In many classical learning problems, it is appropriate to represent input training data as probability distributions instead of single points. Given a finite set of probability measures $\mathbb{P}_i$ on a measurable space $(\mathcal{X},\mathcal{A})$, where $\mathcal{A}$ is a $\sigma$-algebra of subsets of $\mathcal{X}$, introduce the expectations
\[
L_i w = E_{\mathbb{P}_i}(w)= \int_{\mathcal{X}} w(x) d\mathbb{P}_i(x).
\]
Then, given output labels $y_i$, one can learn a input-output relationship by solving regularization problems of the form
\[
\min_{w\in\mathcal{H}}\left(c\left((y_1,E_{\mathbb{P}_1}(w)),\cdots, (y_{\ell},E_{\mathbb{P}_{\ell}}(w)\right)+\|w\|^2\right).
\]
If the expectations are bounded linear functionals, such regularization functional is of the form (\ref{E01}).
\end{exa}

\begin{exa}[Ivanov regularization]
By allowing the regularizer $\Omega$ to take the value $+\infty$, we can also take into account the whole class of Ivanov-type regularization problems of the form
\[
\min_{w \in \mathcal{H}} f(L_1 w, \ldots, L_{\ell} w ), \quad \textrm{ subject to } \quad \phi(w) \leq 1,
\]
by reformulating them as the minimization of a functional of the type (\ref{E01}), where
\[
\Omega(w) = \left\{
              \begin{array}{ll}
                0, & \phi(w) \leq 1 \\
                +\infty, & otherwise
              \end{array}
            \right..
\]
\end{exa}

Let's now go back to the general formulation (\ref{E01}). By the Riesz representation theorem \cite{Riesz07, Frechet07}, $J$ can be rewritten as
\[
J(w) = f(\langle w, w_1 \rangle, \ldots, \langle w, w_{\ell} \rangle) + \Omega(w),
\]
\noindent where $w_i$ is the representer of the linear functional $L_i$ with respect to the inner product. Consider the following definition.

\begin{mydef} \label{DEF01}
A family $\mathcal{F}$ of regularization functionals of the form (\ref{E01}) is said to admit a \emph{linear representer theorem} if, for any $J \in \mathcal{F}$, and any choice of bounded linear functionals $L_i$, there exists a minimizer $w^*$ that can be written as a linear combination of the representers:
\[
w^* = \sum_{i=1}^{\ell}c_i w_i.
\]
\end{mydef}

If a linear representer theorem holds, the regularization problem boils down to a $\ell$-dimensional optimization problem on the scalar coefficients $c_i$. This property is important in practice, since it allows to employ numerical optimization techniques to compute a solution, independently of the dimension of $\mathcal{H}$. Sufficient conditions under which a family of functionals admits a representer theorem have been widely studied in the literature of statistics, inverse problems, and machine learning. The theorem also provides the foundations of learning techniques such as regularized kernel methods and support vector machines, see \cite{Vapnik98, Scholkopf01b, Shawe-Taylor04} and references therein.

Representer theorems are of particular interest when $\mathcal{H}$ is a reproducing kernel Hilbert space (RKHS) \cite{Aronszajn50}. Given a non-empty set $\mathcal{X}$, a RKHS is a space of functions $w:\mathcal{X}\rightarrow \mathbb{R}$ such that point-wise evaluation functionals are bounded, namely, for any $x \in \mathcal{X}$, there exists a non-negative real number $C_x$ such that
\[
|w(x)| \leq C_x \|w\|, \quad \forall w \in \mathcal{H}.
\]
\noindent It can be shown that a RKHS can be uniquely associated to a positive-semidefinite kernel function $K:\mathcal{X}\times\mathcal{X}\rightarrow \mathbb{R}$ (called \emph{reproducing kernel}), such that so-called \emph{reproducing property} holds:
\[
w(x) = \langle w, K_x \rangle, \qquad \forall \left(x,w\right) \in \mathcal{X} \times \mathcal{H},
\]
\noindent where the \emph{kernel sections} $K_x$ are defined as
\[
K_x(y) = K(x,y), \qquad \forall y \in \mathcal{X}.
\]
\noindent The reproducing property states that the representers of point-wise evaluation functionals coincide with the kernel sections. Starting from the reproducing property, it is also easy to show that the representer of any bounded linear functional $L$ is given by a function $K_L \in \mathcal{H}$ such that
\[
K_L(x) = L K_x, \qquad  \forall x \in \mathcal{X}.
\]
\noindent Therefore, in a RKHS, the representer of any bounded linear functional can be obtained explicitly in terms of the reproducing kernel.

If the regularization functional (\ref{E01}) admits minimizers, and the regularizer $\Omega$ is a nondecreasing function of the norm, i.e.
\begin{equation}\label{E02}
\Omega(w) = h(\|w\|), \quad \textup{ with } h:\mathbb{R}\rightarrow \mathbb{R}\cup\{+\infty\}, \textup{ nondecreasing,}
\end{equation}
\noindent the linear representer theorem follows easily from the Pythagorean identity. A proof that the condition (\ref{E02}) is sufficient appeared in \cite{Scholkopf01} in the case where $\mathcal{H}$ is a RKHS and $L_i$ are point-wise evaluation functionals. Earlier instances of representer theorems can be found in \cite{Kimeldorf71, Cox90, Poggio90}. More recently, the question of whether condition (\ref{E02}) is also necessary for the existence of linear representer theorems has been investigated \cite{Argyriou09}. In particular, \cite{Argyriou09} shows that, if $\Omega$ is differentiable (and certain technical existence conditions hold), then (\ref{E02}) is necessary and sufficient. The proof of \cite{Argyriou09} heavily exploits differentiability of $\Omega$, but the authors conjecture that the hypothesis can be relaxed. In this report, we show that (\ref{E02}) is necessary and sufficient for the family of regularization functionals of the form (\ref{E01}) to admit a linear representer theorem, by merely assuming that $\Omega$ is lower semicontinuous and satisfies basic conditions for the existence of minimizers. The proof is based on a characterization of radial nondecreasing functionals on a Hilbert space.

\section{A characterization of radial nondecreasing functionals}

In this section, we present a characterization of radial nondecreasing functionals defined over Hilbert spaces. We will make use of the following definition.

\begin{mydef}
A subset $\mathcal{S}$ of a Hilbert space $\mathcal{H}$ is called \emph{star-shaped} with respect to a point $z \in \mathcal{H}$ if
\[
(1-\lambda) z + \lambda x \in \mathcal{S}, \quad \forall x \in \mathcal{S}, \quad \forall \lambda \in [0, 1].
\]
\end{mydef}

It is easy to verify that a convex set is star-shaped with respect to any point of the set, whereas a star-shaped set does not have to be convex.

The following Theorem provides a geometric characterization of radial nondecreasing functions defined on a Hilbert space that generalizes the analogous result of \cite{Argyriou09} for differentiable functions.

\begin{thm}\label{THM1}
Let $\mathcal{H}$ denote a Hilbert space such that $\dim{\mathcal{H}} \geq 2$, and let $\Omega:\mathcal{H} \rightarrow \mathbb{R}\cup \{+\infty\}$ a lower semicontinuous function. Then, (\ref{E02}) holds if and only if
\begin{equation}\label{E04}
\Omega(x+y) \geq \max\{\Omega(x),\Omega(y)\}, \qquad \forall x, y \in \mathcal{H}: \langle x, y\rangle = 0.
\end{equation}
\end{thm}

\begin{proof}

Assume that (\ref{E02}) holds. Then, for any pair of orthogonal vectors $x, y \in \mathcal{H}$, we have
\begin{align*}
\Omega(x+y)  & = h\left(\|x+y\|\right) = h\left(\sqrt{\|x\|^2+\|y\|^2}\right) \geq \max\{h\left(\|x\|\right),h\left(\|y\|\right)\}\\
             & = \max\{\Omega(x),\Omega(y)\}.
\end{align*}

\noindent Conversely, assume that condition (\ref{E04}) holds. Since $\dim{\mathcal{H}} \geq 2$, by fixing a generic vector $x \in \mathcal{X} \setminus \{0\}$ and a number $\lambda \in [0, 1]$, there exists a vector $y$ such that $\|y\| = 1$ and
\[
\lambda =  1-\cos^2 \theta,
\]
\noindent where
\[
\cos \theta = \frac{\langle x, y \rangle}{\|x\|\|y\|}.
\]

\noindent In view of (\ref{E04}), we have
\begin{align*}
\Omega(x)  & = \Omega(x- \langle x, y \rangle y + \langle x, y \rangle y)\\
      & \geq \Omega(x- \langle x, y \rangle y) = \Omega\left(x- \cos^2 \theta x+ \cos^2 \theta x - \langle x, y \rangle y \right)\\
      & \geq \Omega\left(\lambda x\right).
\end{align*}
\noindent Since the last inequality trivially holds also when $x = 0$, we conclude that
\begin{equation}\label{E05}
\Omega(x) \geq \Omega(\lambda x), \qquad \forall x \in \mathcal{H},\quad \forall \lambda \in [0, 1],
\end{equation}
\noindent so that $\Omega$ is non-decreasing along all the rays passing through the origin. In particular, the minimum of $\Omega$ is attained at $x = 0$.

Now, for any $c \geq \Omega(0)$, consider the sublevel sets
\[
\mathcal{S}_c = \left\{ x \in \mathcal{H}: \Omega(x) \leq c \right\}.
\]
\noindent From (\ref{E05}), it follows that $\mathcal{S}_c$ is not empty and star-shaped with respect to the origin. In addition, since $\Omega$ is lower semi-continuous, $\mathcal{S}_c$ is also closed. We now show that $\mathcal{S}_c$ is either a closed ball centered at the origin, or the whole space. \noindent To this end, we show that, for any $x \in \mathcal{S}_c$, the whole ball
\[
\mathcal{B} =\{y\in\mathcal{H}: \|y\| \leq \|x\|\},
\]
is contained in $\mathcal{S}_c$. First, take any $y \in \textrm{int}(\mathcal{B}) \setminus \textrm{span}\{x\}$, where \textrm{int} denotes the interior. Then, $y$ has norm strictly less than $\|x\|$, that is
\[
0 < \|y\| < \|x\|,
\]
\noindent and is not aligned with $x$, i.e.
\[
y \neq \lambda x, \quad \forall \lambda \in \mathbb{R}.
\]
\noindent Let $\theta \in \mathbb{R}$ denote the angle between $x$ and $y$. Now, construct a sequence of points $x_k$ as follows:
\[
\left\{
  \begin{array}{ll}
    x_0 = y, &  \\
    x_{k+1} = x_k + a_k u_k,
  \end{array}
\right.
\]
\noindent where
\[
a_k = \|x_k\| \tan\left(\frac{\theta}{n}\right), \qquad n \in \mathbb{N}
\]
\noindent and $u_k$ is the unique unitary vector that is orthogonal to $x_k$, belongs to the two-dimensional subspace $\textrm{span}\{x,y\}$, and is such that $\langle u_k, x \rangle > 0$, that is
\[
u_k \in \textrm{span}\{x,y\}, \qquad \|u_k\| = 1, \qquad \langle u_k, x_k \rangle = 0, \qquad \langle u_k, x \rangle > 0.
\]

\noindent By orthogonality, we have
\begin{equation}\label{E09}
\|x_{k+1}\|^2 = \|x_k\|^2 + a_k^2 = \|x_k\|^2\left(1+ \tan^2\left(\frac{\theta}{n}\right)\right) = \|y\|^2\left(1+ \tan^2\left(\frac{\theta}{n}\right)\right)^{k+1}.
\end{equation}
\noindent In addition, the angle between $x_{k+1}$ and $x_k$ is given by
\[
\theta_k = \arctan\left(\frac{a_k}{\|x_k\|}\right) = \frac{\theta}{n},
\]
\noindent so that the total angle between $y$ and $x_n$ is given by
\[
\sum_{k=0}^{n-1} \theta_k = \theta.
\]
\noindent Since all the points $x_k$ belong to the subspace spanned by $x$ and $y$, and the angle between $x$ and $x_n$ is zero, we have that $x_n$ is positively aligned with $x$, that is
\[
x_n = \lambda x, \qquad \lambda \geq 0.
\]
Now, we show that $n$ can be chosen in such a way that $\lambda \leq 1$. Indeed, from (\ref{E09}) we have
\[
\lambda^2 = \left(\frac{\|x_n\|}{\|x\|}\right)^2 = \left(\frac{\|y\|}{\|x\|}\right)^2 \left(1+ \tan^2\left(\frac{\theta}{n}\right)\right)^n,
\]
\noindent and it can be verified that
\[
\lim_{n \rightarrow +\infty} \left(1+ \tan^2\left(\frac{\theta}{n}\right)\right)^n = 1,
\]
\noindent therefore $\lambda \leq 1$ for a sufficiently large $n$. Now, write the difference vector in the form
\[
\lambda x-y = \sum_{k=0}^{n-1} (x_{k+1}-x_k),
\]
\noindent and observe that
\[
\langle x_{k+1}-x_k, x_k \rangle = 0.
\]
By using (\ref{E05}) and proceeding by induction, we have
\[
c \geq \Omega(\lambda x) = \Omega\left(x_n-x_{n-1}+x_{n-1}\right) \geq \Omega(x_{n-1}) \geq \cdots \geq \Omega(x_0) = \Omega(y),
\]
\noindent so that $y \in \mathcal{S}_c$.  Since $\mathcal{S}_c$ is closed and the closure of $\textrm{int}(\mathcal{B}) \setminus \textrm{span}\{x\}$ is the whole ball $\mathcal{B}$, every point $y \in \mathcal{B}$ is also included in $\mathcal{S}_c$. This proves that $\mathcal{S}_c$ is either a closed ball centered at the origin, or the whole space $\mathcal{H}$.

Finally, for any pair of points such that $\|x\| = \|y\|$, we have $x \in \mathcal{S}_{\Omega(y)}$, and $y \in \mathcal{S}_{\Omega(x)}$, so that
\[
\Omega(x) = \Omega(y).
\]

\end{proof}

\section{Representer theorem: a necessary and sufficient condition}

In this section, we  prove that condition (\ref{E02}) is necessary and sufficient for suitable families of regularization functionals of the type (\ref{E01}) to admit a linear representer theorem.

\begin{thm}\label{THM3}
Let $\mathcal{H}$ denote a Hilbert space such that $\dim{\mathcal{H}} \geq 2$. Let $\mathcal{F}$ denote a family of functionals $J:\mathcal{H} \rightarrow \mathbb{R} \cup \{+\infty\}$ of the form (\ref{E01}) that admit minimizers.

\begin{enumerate}
  \item If $\Omega$ satisfy (\ref{E02}), then $\mathcal{F}$ admits a linear representer theorem.
  \item Conversely, assume that $\mathcal{F}$ contains a set of functionals of the form
\begin{equation}\label{EQ11}
J_{p}^{\gamma}(w) = \gamma f\left(\left\langle w, p\right\rangle\right) + \Omega\left(w\right), \qquad \forall p \in \mathcal{H}, \quad \forall \gamma \in \mathbb{R}_+,
\end{equation}
\noindent where $f(z)$ is uniquely minimized at $z =1$. For any lower-semicontinuous $\Omega$, the family $\mathcal{F}$ admits a linear representer theorem only if (\ref{E02}) holds.
\end{enumerate}
\end{thm}

\begin{proof} The first part of the theorem (sufficiency) follows from an orthogonality argument. Take any functional $J \in \mathcal{F}$. Let $\mathcal{R} = \textrm{span}\{w_1, \ldots, w_{\ell}\}$ and let $\mathcal{R}^{\perp}$ denote its orthogonal complement. Any minimizer $w^*$ of $J$ can be uniquely decomposed as
\[
w^* = u + v, \qquad u \in \mathcal{R}, \quad v \in \mathcal{R}^{\perp}.
\]
\noindent If (\ref{E02}) holds, then we have
\[
J(w^*) -J(u) = h(\|w^*\|)-h(\|u\|) \geq 0,
\]
\noindent so that $u \in \mathcal{R}$ is also a minimizer.

Now, let's prove the second part of the theorem. First of all, observe that the functional
\[
J_0^{\gamma}(w) = \gamma f(0)+\Omega(w),
\]
\noindent obtained by setting $p =0$ in (\ref{EQ11}), belongs to $\mathcal{F}$. By hypothesis, $J_0^{\gamma}$ admits minimizers. In addition, by the representer theorem, the only admissible minimizer of $J_0$ is the origin, that is
\begin{equation}\label{E07}
\Omega(y) \geq \Omega(0), \qquad \forall y \in \mathcal{H}.
\end{equation}

Now take any $x \in \mathcal{H} \setminus \{0\}$ and let
\[
p = \frac{x}{\|x\|^2}.
\]

By the representer theorem, the functional $J_p^{\gamma}$ of the form (\ref{EQ11}) admits a minimizer of the type
\[
w = \lambda(\gamma) x.
\]
\noindent Now, take any $y \in \mathcal{H}$ such that $\langle x, y \rangle = 0$. By using the fact that $f(z)$ is minimized at $z =1$, and the linear representer theorem, we have
\[
\gamma f(1) + \Omega\left(\lambda(\gamma) x\right) \leq  \gamma f(\lambda(\gamma)) + \Omega\left(\lambda(\gamma) x\right) = J_p^{\gamma}(\lambda(\gamma)x) \leq J_p^{\gamma}(x + y) = \gamma f(1) + \Omega\left(x+y\right).
\]
\noindent By combining this last inequality with (\ref{E07}), we conclude that
\begin{equation}\label{E06}
\Omega\left(x+y\right) \geq \Omega\left(\lambda(\gamma)x\right), \qquad \forall x, y \in \mathcal{H}: \langle x, y\rangle = 0, \qquad \forall \gamma \in \mathbb{R}_+.
\end{equation}

\noindent Now, there are two cases:
\begin{itemize}
  \item $\Omega\left(x+y\right) = +\infty$
  \item $\Omega\left(x+y\right) = C < +\infty$.
\end{itemize}
In the first case, we trivially have
\[
\Omega\left(x+y\right) \geq \Omega(x).
\]
\noindent In the second case, using (\ref{E07}) and (\ref{E06}), we obtain
\begin{equation}\label{EQ10}
0 \leq \gamma\left(f(\lambda(\gamma))-f(1)\right) \leq \Omega\left(x+y\right) - \Omega\left(\lambda(\gamma)x\right) \leq C-\Omega(0) < +\infty, \qquad \forall \gamma \in \mathbb{R}_+.
\end{equation}

\noindent Let $\gamma_k$ denote a sequence such that $\lim_{k\rightarrow +\infty} \gamma_k = +\infty$, and consider the sequence
\[
a_k = \gamma_k\left(f(\lambda(\gamma_k))-f(1)\right).
\]

\noindent From (\ref{EQ10}), it follows that $a_k$ is bounded. Since $z =1$ is the only minimizer of $f(z)$, the sequence $a_k$ can remain bounded only if
\[
\lim_{k \rightarrow +\infty} \lambda(\gamma_{k}) = 1.
\]

\noindent By taking the limit inferior in (\ref{E06}) for $\gamma \rightarrow +\infty$, and using the fact that $\Omega$ is lower semicontinuous, we obtain condition (\ref{E04}). It follows that $\Omega$ satisfies the hypotheses of Theorem \ref{THM1}, therefore (\ref{E02}) holds.
\end{proof}

The second part of Theorem \ref{THM3} states that any lower-semicontinuous regularizer $\Omega$ has to be of the form (\ref{E02}) in order for the family $\mathcal{F}$ to admit a linear representer theorem. Observe that $\Omega$ is not required to be differentiable or even continuous. Moreover, it needs not to have bounded lower level sets. For the necessary condition to holds, the family $\mathcal{F}$ has to be broad enough to contain at least a set of regularization functionals of the form (\ref{EQ11}). The following examples show how to apply the necessary condition of Theorem \ref{THM3} to classes of regularization problems with standard loss functions.

\begin{itemize}
  \item Let $L:\mathbb{R}^2 \rightarrow \mathbb{R} \cup \{+\infty\}$ denote any loss function of the type
  \[
  L(y,z) = \widetilde{L}(y-z),
  \]
  \noindent such that $\widetilde{L}(t)$ is uniquely minimized at $t = 0$. Then, for any lower-semicontinuous regularizer $\Omega$, the family of regularization functionals of the form
  \[
  J(w) = \gamma \sum_{i=1}^{\ell}L\left(y_i,\langle w, w_i\rangle\right) +\Omega(w),
  \]
  admits a linear representer theorem if and only if (\ref{E02}) holds. To see that the hypotheses of Theorem \ref{THM3} are satisfied, it is sufficient to consider the subset of functionals with $\ell =1$, $y_1 = 1$, and $w_1 = p \in \mathcal{H}$. These functionals can be written in the form (\ref{EQ11}) with
  \[
  f(z) = L(1,z).
  \]

  \item The class of regularization problems with the hinge (SVM) loss of the form
  \[
  J(w) = \gamma \sum_{i=1}^{\ell}\max\{0, 1-y_i \langle w, w_i\rangle\} +\Omega(w),
  \]
  with $\Omega$ lower-semicontinuous, admits a linear representer theorem if and only if $\Omega$ satisfy (\ref{E02}). For instance, by choosing $\ell =2$, and
  \[
  \left(y_1, w_1 \right) = (1, p), \qquad (y_2, w_2) = (-1, p/2),
  \]
  we obtain regularization functionals of the form (\ref{EQ11}) with
  \[
  f(z) = \max\{0,1-z\} + \max\{0, 1+z/2\},
  \]
  \noindent and it is easy to verify that $f$ is uniquely minimized at $z =1$.
\end{itemize}

\section{Conclusions}

We have shown that some general families of regularization functionals defined over a Hilbert space with lower semicontinuous regularizer admits a linear representer theorem if and only if the regularizer is a radial nondecreasing function. The result extends a previous characterization of \cite{Argyriou09}, by relaxing the assumptions on the regularization term. We provide a unified proof that holds simultaneously for the finite and the infinite dimensional case.

\end{document}